\documentclass[12pt]{article}

\usepackage{amsmath} %Advanced math typing
\usepackage{amsthm} %For theorem styles
\usepackage{amsfonts} %For \mathbb
\usepackage{amssymb} %For more math symbols
\usepackage{stmaryrd} %For even more math symbols (like \llbracket)
\usepackage{mathtools} %For extensible math symbols

\usepackage{bm} %For bold symbols with \bm
\usepackage{bbm} %For alternative blackboardsymbols with \mathbbm

\usepackage{array} %For special commands in tables
\usepackage[inline]{enumitem} %For enumerate styles
\usepackage{microtype}

\usepackage{tikz} %For tikzfigures
\usetikzlibrary{calc} %For coordinate calculations

\usepackage{subcaption} % For captions in subfloats

\usepackage[%
  pdftex,% For direct pdf compilation compatibility
  plainpages=false,% Usage of not only arabic numbers
  pdfpagelabels,% Enable page labels
  colorlinks,% Colored links
  citecolor=black,% Citation color
  linkcolor=black,% Link color
  urlcolor=black,% URL color
  filecolor=black,% File color
  bookmarksopen=false% For starting document with bookmarks tree closed
]{hyperref} %For pdf links
\usepackage[all]{hypcap} %Fixes figures and tables anchors for hyperref

\newtheoremstyle{thmstyle}%style name
  {\medskipamount}%space before
  {\smallskipamount}%space after
  {\slshape}%font used
  {0pt}%indentation
  {\bfseries}%modifier theorem head
  {.}%punctuation between theorem head and body
  { }%space after punctuation
  {\thmname{#1}\thmnumber{ #2}{\normalfont\thmnote{ (#3)}}}%theorem specifier

\newtheoremstyle{plainstyle}%style name
  {\medskipamount}%space before
  {\smallskipamount}%space after
  {\rmfamily}%font used
  {0pt}%indentation
  {\bfseries}%modifier theorem head
  {.}%punctuation between theorem head and body
  { }%space after punctuation
  {\thmname{#1}\thmnumber{ #2}{\normalfont\thmnote{ (#3)}}}%theorem specifier

\theoremstyle{thmstyle}
\newtheorem{theorem}{Theorem}[section]
\newtheorem{lemma}[theorem]{Lemma}
\newtheorem{corollary}[theorem]{Corollary}

\newtheorem{claim}[theorem]{Claim}

\theoremstyle{plainstyle}
\newtheorem{definition}[theorem]{Definition}
\newtheorem{remark}{Remark}
\newtheorem{conjecture}{Conjecture}

\newenvironment{proofof}[1]{\begin{proof}[Proof of #1.]}{\end{proof}}

\setlist[enumerate]{label={\roman*.}, ref={(\roman*)}}

\newcommand{\df}{\stackrel{\text{def}}{=}}

\newcommand{\comp}{\mathbin{\circ}}
\newcommand{\rest}{\mathord{\vert}}

\newcommand{\function}[2]{\colon #1 \rightarrow #2}
\newcommand{\injection}[2]{\colon #1 \rightarrowtail #2}

\newcommand{\set}[2]{\left\{\hspace{0.2ex} #1 \:\middle\vert\: #2\right\}}
\DeclareMathOperator{\Hom}{Hom}
\DeclareMathOperator{\im}{im}

\DeclareMathOperator*{\esssup}{ess\ sup}
\DeclareMathOperator*{\essinf}{ess\ inf}
\def\dave{d_{\operatorname{ave}}}

\newcommand{\NN}{\mathbb{N}}

\newcommand{\RR}{\mathbb{R}}

\newcommand{\One}{\mathbbm{1}}

\newcommand{\cW}{\mathcal{W}}

\title{Biregularity in Sidorenko's Conjecture}
\author{Leonardo N.~Coregliano\thanks{University of Chicago, {\tt lenacore@uchicago.edu}} \and Alexander A.~Razborov\thanks{University of Chicago,
{\tt razborov@math.uchicago.edu}, and Steklov Mathematical Institute, Moscow}
}
\begin{document}
\maketitle

\begin{abstract}
  Sidorenko's Conjecture says that the minimum density of a bigraph $G$ in
  a bigraphon $W$ of a given edge density is attained when $W$ is a
  constant function. A consequence of a result by B.~Szegedy is that it is
  enough to show Sidorenko's Conjecture under the further assumption that
  $W$ is biregular. In this paper, we retrieve this result with a more
  elementary proof. With this biregularity result and some ideas of its
  proof, we also obtain simple proofs of several other results related to
  Sidorenko's Conjecture. Furthermore, we also show that bigraphs that have
  a special type of tree decomposition, called reflective tree
  decomposition, satisfy Sidorenko's conjecture. This both unifies and
  generalizes the notions of strong tree decompositions and
  $N$-decompositions from the literature.
\end{abstract}

\section{Introduction}

In~\cite{Sid91} (see also~\cite{Sid93}), Sidorenko conjectured that if
$\Omega=(X,\mu)$ and $\Lambda=(Y,\nu)$ are probability spaces, $W\function{X\times Y}{\RR_+}$ is a bounded measurable
function (a \emph{bigraphon}), and $G=(V_1,V_2,E)$ is a bipartite graph with a given bipartition
$(V_1,V_2)$ (a \emph{bigraph}), then
\begin{align}\label{eq:Sidorenkooriginal}
  t(G,W) & \geq t(\rho,W)^{e(G)},
\end{align}
where $\rho$ denotes the bigraph consisting of a single edge, $e(G)\df\lvert
E(G)\rvert$ is the number of edges in $G$ and the density of $G$ in $W$ is
naturally defined as
\begin{align}\label{eq:tGW}
  t(G,W)
  & \df
  \int_{X^{V_1}\times Y^{V_2}} \prod_{(v,w)\in E(G)} W(x_v,y_w)\ d(\mu^{V_1}\otimes\nu^{V_2})(x,y).
\end{align}
Bigraphs $G$ that satisfy~\eqref{eq:Sidorenkooriginal} for every $W$ are called \emph{Sidorenko bigraphs}.

A tensor power trick~\cite[Remark~2]{Sid91} implies that to show that $G$ is a Sidorenko bigraph, it is
sufficient to prove that there exists $c_G > 0$ such that for every $W$, we have
\begin{align}\label{eq:Sidorenkomult}
  t(G,W) & \geq c_G\cdot t(\rho,W)^{e(G)}.
\end{align}
With this tensor power trick, Sidorenko showed~\cite[Theorem~10]{Sid91} that
his conjecture is equivalent to a conjecture by
Simonovits~\cite[Conjecture~8]{Sim84}. Since Simonovits's Conjecture is a
weak version of another joint conjecture with
Erd\H{o}s~\cite[Conjecture~2]{ES84} on supersaturation (see
also~\cite[Conjecture~7]{Sim84}), Sidorenko's Conjecture is sometimes
referred to as Erd\H{o}s--Sidorenko--Simonovits Conjecture, possibly with
some permutation of these names. While quite easy to prove (see
Lemma~\ref{lem:tensor} below for a general version), this tensor power trick
has been essential to several results in the literature,
see~\cite{CFS10,LS11,KLL16,CKLL18a,CKLL18b,Sid21} for some examples.

From an easy adaptation of the graphon theory (see~\cite{Lov12} for an introduction to the topic)
to the case of bigraphs, \eqref{eq:Sidorenkomult}
is in turn equivalent to
\begin{align}\label{eq:Sidorenkobigraph}
  t(G,H) & \geq c_G\cdot t(\rho,H)^{e(G)}
\end{align}
for every bigraph $H=(U_1,U_2,F)$, where
\begin{align*}
  t(G,H) & \df \frac{\lvert\Hom(G,H)\rvert}{\lvert U_1\rvert^{\lvert V_1\rvert}\cdot\lvert U_2\rvert^{\lvert V_2\rvert}}
\end{align*}
and $\Hom(G,H)$ is the set of all bigraph homomorphisms from $G$ to $H$, i.e., functions $f\function{V_1\cup
  V_2}{U_1\cup U_2}$ such that $f(V_i)\subseteq U_i$ ($i=1,2$) and $(v,w)\in E\implies (f(v),f(w))\in F$. In
fact, Sidorenko's Conjecture is often studied under the further assumption that $W$ is symmetric (i.e.,
$\Omega=\Lambda$ and $W(x,y)=W(y,x)$ for every $x,y\in X$), in which case $G$ and $H$ become ordinary graphs
(but $G$ is still bipartite).

It was proved in~\cite[Theorem~4]{Sze15b} that it is sufficient to show~\eqref{eq:Sidorenkobigraph} under the
further assumption that $H$ is edge-vertex transitive (i.e., the natural actions of the automorphism group of
$H$ on the sets $V_1(H)$, $V_2(H)$ and $E(H)$ are transitive). In this paper we recover an important consequence of this
result through a different, more elementary, method. More specifically, we prove
(Theorem~\ref{thm:biregularity}) that in order to show that a bigraph is Sidorenko, it is sufficient to
show~\eqref{eq:Sidorenkomult} under the additional assumption that $W$ is biregular in the sense
\begin{equation} \label{eq:biregularity}
  \int_Y W(x_0,y)\ d\nu(y) = \int_X W(x,y_0)\ d\mu(x) = t(\rho,W)
\end{equation}
for almost every $x_0\in X$ and almost every $y_0\in Y$.

Our techniques and the biregularity assumption allow us to both retrieve
several results from the literature in the non-symmetric setting with a much simpler proof and provide
some generalizations.

For example, not only can we obtain an easy proof of the non-symmetric analogue of a result of~\cite[Lemmas~3.2 and~3.4 and Theorem~5]{LS11}
that amalgamations of Sidorenko bigraphs along a vertex are Sidorenko bigraphs
(Theorem~\ref{thm:amalgamation}), but we can prove a weak converse: if $G'$ is the amalgamation of $k$ copies
of $G$ along the same vertex, then $G$ is Sidorenko if and only if $G'$ is Sidorenko
(Theorem~\ref{thm:power}). Of course, these two results also follow from~\cite{Sze15b}.

When studying~\eqref{eq:Sidorenkooriginal}, Sidorenko in fact introduced a
stronger (a priori) conjecture~\cite[Equation~(2)]{Sid91} that in particular
implies
\begin{align*}
  t(G,W)
  & \geq
  \left(\int_X
  \left(\int_Y W(x,y)\ d\nu(y)\right)^{e(G)/\lvert V_1(G)\rvert}
  \ d\mu(x)\right)^{\lvert V_1(G)\rvert}.
\end{align*}
One way of interpreting the right-hand side is as $t(K_{1,d},W)^{e(G)/d}$, where $K_{1,d}$ is the (left)
$d$-star bigraph, except that in the above $d\df e(G)/\lvert V_1(G)\rvert$, which is not necessarily an
integer. Our methods allow us to retrieve a weaker version of this implication in the ``ordinary'' setting; namely, we
show that every Sidorenko bigraph $G=(V_1,V_2,E)$ in which all vertices of $V_1$ have degree at least $d$
satisfies $t(G,W)\geq t(K_{1,d},W)^{e(G)/d}$ (Theorem~\ref{thm:stars}).

Finally, by using the biregularity assumption we are able to unify and generalize (Theorem~\ref{thm:reftree})
the results of~\cite[Theorem~1.2]{CKLL18a} on strong tree decompositions and of~\cite[Theorem~5.12]{CL17} on
$N$-decompositions as particular cases of what we call reflective tree decompositions (see
Definition~\ref{def:reftree}); this result holds in both the non-symmetric and symmetric settings.

\medskip

This paper is organized as follows. In Section~\ref{sec:prelim} we establish
necessary notation. In Section~\ref{sec:results} we state our results. In
Section~\ref{sec:mainlemma} we present the main lemma used in the proofs. In
Section~\ref{sec:biregularity} we prove our biregularity result. In
Section~\ref{sec:applications}, we prove the aforementioned applications of
our biregularity result and our main lemma to amalgamations and the
strengthened $K_{1,d}$ version of Sidorenko's Conjecture. In
Section~\ref{sec:reftree} we prove the result on reflective tree
decompositions. In Section~\ref{sec:symmetric} we show how to adapt the
material from Sections~\ref{sec:mainlemma}, \ref{sec:biregularity}
and~\ref{sec:reftree} to the symmetric setting. We finish the paper with a
brief discussion and some open problems in Section~\ref{sec:conclusion}.

\section{Preliminaries}
\label{sec:prelim}

Throughout the text, we will use the notation $\NN\df\{0,1,\ldots\}$ for non-negative integers and
$\NN_+\df\NN\setminus\{0\}$ for positive integers. For $n\in\NN$, we let $[n]\df\{1,\ldots,n\}$. We also
let $\RR$ be the set of real numbers and $\RR_+$ the set of non-negative real numbers. Given a set $V$, we denote its power
set by $2^V\df\set W{W\subseteq V}$.

\subsection{Bigraphs}

A \emph{bigraph} is a triple $G=(V_1,V_2,E)$, where $V_1$ and $V_2$ are
disjoint finite sets and $E\subseteq V_1\times V_2$. We will also use the following notation ($i=1,2)$:
\begin{align*}
  V_i(G) & \df V_i, &
  v_i(G) & \df \lvert V_i\rvert, &
  V(G) & \df V_1\cup V_2,
  \\
  E(G) & \df E, &
  e(G) & \df \lvert E\rvert, &
  v(G) & \df \lvert V_1\rvert + \lvert V_2\rvert.
\end{align*}
For $v\in V(G)$, we denote by $d_G(v)$ its degree.
We also let
\begin{align*}
  \delta_i(G) & \df \min_{v\in V_i(G)} d_G(v), &
  \Delta_i(G) & \df \max_{v\in V_i(G)} d_G(v).
\end{align*}
We say that $G$ is \emph{left $d$-regular} (\emph{right $d$-regular}, respectively) if $d_G(v) = d$ for every
$v\in V_1(G)$ ($v\in V_2(G)$, resp.). We say that $G$ is \emph{biregular} if it is both left $d_1$-regular and
right $d_2$-regular for some $d_1,d_2\in\NN$. An \emph{isomorphism} between bigraphs $G_1$ and $G_2$ is a
bijection $f\injection{V(G_1)}{V(G_2)}$ such that $f(V_i(G_1))=V_i(G_2)$ ($i=1,2$) and $(v,w)\in E(G_1)\iff
(f(v),f(w))\in E(G_2)$ ($(v,w)\in V_1(G_1)\times V_2(G_1)$); when such an isomorphism exists, we say that
$G_1$ and $G_2$ are \emph{isomorphic} and denote this as $G_1\cong G_2$.

For $U\subseteq V(G)$, we let $G\rest_U$ be the \emph{subgraph induced by $U$
in $G$}, that is, we let
\begin{align*}
  V_i(G\rest_U) & \df V_i(G)\cap U, &
  E(G\rest_U) & \df E(G)\cap ((U\cap V_1(G))\times (U\cap V_2(G))).
\end{align*}
For $v\in V(G)$, we let $G-v\df G\rest_{V(G)\setminus\{v\}}$ be the bigraph
obtained from $G$ by removing $v$. For $E\subseteq E(G)$, we also let $G -
E\df (V_1(G),V_2(G),E(G)\setminus E)$ be the spanning subgraph obtained from $G$ by removing
the edges in $E$. The \emph{dual bigraph} of $G$ is the bigraph $G^*\df
(V_2,V_1,E^*)$, where $E^*\df\{(w,v)\mid (v,w)\in E(G)\}$. We denote the
\emph{edge bigraph} $(\{1\},\{2\},\{(1,2)\})$ by $\rho$ and the
\emph{$d$-star bigraph} $(\{0\},[d],\{(0,i)\mid i\in[d]\})$ by $K_{1,d}$ (thus,
$\rho\cong K_{1,1}$).

\subsection{Flags}

It will be convenient to also work with partially labeled bigraphs and for this purpose we will borrow some
terminology from the theory of flag algebras~\cite{Raz07}.

More specifically, we work in the theory $T_{\operatorname{Graph}}^2$ of
graphs augmented with a 2-coloring of its vertices. Thus, a \emph{flag} is a
partially labeled bigraph, that is, a pair $F=(G,\theta)$, where $G$
is a bigraph and $\theta\injection{[k]}{V(G)}$ is an injection for some
$k\in\NN$. We use the notation $\lvert F\rvert\df G$ for the \emph{underlying bigraph}
of $F$ and the notation $\theta_F\df\theta$ for the \emph{labeling} of $F$.
We will often abuse notation and write
$F=(G,(\theta(1),\theta(2),\ldots,\theta(k)))$, listing the values of
$\theta$. In fact, we will abuse the notation even more and write $F=(G,U)$
for some set $U\subseteq V(G)$ to be understood as $F=(G,\theta)$ for some
$\theta\injection{[\lvert U\rvert]}{V(G)}$ with $\im(\theta)=U$, whenever the
exact ordering is either clear from the context or unimportant.

An \emph{isomorphism} between flags $F_1=(G_1,\theta_1)$ and $F_2=(G_2,\theta_2)$ is an isomorphism $f$
between $G_1$ and $G_2$ that preserves the partial labeling in the sense that $f\comp\theta_1=\theta_2$;
when such an isomorphism exists, we say that
$F_1$ and $F_2$ are \emph{isomorphic} and denote it by $F_1\cong F_2$.

If $F_1=(G_1,\theta_1)$ and $F_2=(G_2,\theta_2)$ are flags such that $\theta_2\comp\theta_1^{-1}$ is an
isomorphism between $G_1\rest_{\im(\theta_1)}$ and $G_2\rest_{\im(\theta_2)}$
(that is, in the terminology of flag algebras, these flags are of the same type), we let
$F_1\sqcup F_2$ be the flag obtained from the disjoint union of $F_1$ and $F_2$ by identifying
vertices with the same label\footnote{We avoid using $F_1F_2$ here to not conflict with the product as defined
  in flag algebras.}. For $k\in\NN_+$, we further let $F^{\sqcup k}$ be defined recursively as
$F^{\sqcup 1}\df F$ and $F^{\sqcup (k+1)}\df F^{\sqcup k}\sqcup F$.

A \emph{left $1$-flag} (\emph{right $1$-flag}, respectively) is a flag
$F=(G,\theta)$ in which $\im(\theta)$ is a single vertex in $V_1(G)$
($V_2(G)$, resp.). We let $e_1\df(\rho,1)$ and $K_{1,d}^L\df(K_{1,d},0)$ be
the unique left $1$-flags such that $\lvert e_1\rvert=\rho$ and $\lvert
K_{1,d}^L\rvert\df K_{1,d}$ (thus $e_1\cong K_{1,1}^L$). We also let
$e_2\df(\rho,2)$ be the unique right $1$-flag such that $\lvert
e_2\rvert=\rho$.

\subsection{Bigraphons}

Given probability spaces $\Omega=(X,\mu)$ and $\Lambda=(Y,\nu)$, a
\emph{bigraphon} over $\Omega$ and $\Lambda$ is a bounded measurable function
$W\function{X\times Y}{\RR_+}$, where $X\times Y$ is equipped with the
product $\sigma$-algebra and the product measure $\mu\otimes\nu$; we will
denote bigraphons by $W\function{\Omega\times\Lambda}{\RR_+}$.

The \emph{dual bigraphon} of $W$ is the bigraphon $W^*\function{\Lambda\times\Omega}{\RR_+}$ defined by
$W^*(y,x)\df W(x,y)$. Given two bigraphons $W_1\function{\Omega_1\times\Lambda_1}{\RR_+}$ and
$W_2\function{\Omega_2\times\Lambda_2}{\RR_+}$, their \emph{tensor product} is the bigraphon $W_1\otimes
W_2\function{(\Omega_1\times\Omega_2)\times(\Lambda_1\times\Lambda_2)}{\RR_+}$ given by $(W_1\otimes
W_2)((x_1,x_2),(y_1,y_2))\df W_1(x_1,y_1)\cdot W_2(x_2,y_2)$. For $k\in\NN_+$, the \emph{$k$th tensor power}
$W^{\otimes k}$ of a bigraphon $W$ is defined inductively by $W^{\otimes 1}\df W$ and $W^{\otimes (k+1)}\df
W^{\otimes k}\otimes W$.

For a bigraphon $W\function{\Omega\times\Lambda}{\RR_+}$ over spaces $\Omega=(X,\mu)$ and $\Lambda=(Y,\nu)$
and measurable sets $X'\subseteq X$, $Y'\subseteq Y$ of positive measure, we let $W\rest_{X'\times
  Y'}\function{\Omega\rest_{X'}\times\Lambda\rest_{Y'}}{\RR_+}$ be the bigraphon that is the restriction of
$W$ to $X'\times Y'$ over the conditional probability spaces $\Omega\rest_{X'}\df(X',\mu\rest_{X'})$ and
$\Lambda\rest_{Y'}\df(Y',\nu\rest_{Y'})$ (that is, their underlying measures are given by
$\mu\rest_{X'}(A)\df\mu(A)/\mu(X')$ and $\nu\rest_{Y'}(B)\df\nu(B)/\nu(Y')$).

\smallskip
When taking integrals, our functions will always be bounded and hence Fubini's Theorem will apply and
we will be omitting explicit references to it. If $V$ is a set, we let $\Omega^V=(X^V,\mu^V)$ be the product
probability space of $\lvert V\rvert$ copies of $\Omega$; we will usually abuse notation and denote $\mu^V$ simply
by $\mu$. Given $x\in X^V$ and $S\subseteq V$, we let $x_S\in X^S$ be the projection of $x$ to the coordinates
in $S$.

For a bigraph $G$ and a bigraphon $W\function{\Omega\times\Lambda}{\RR_+}$,
we let $t(G,W)\in \RR_+$ be given by~\eqref{eq:tGW}. More generally, for a
flag $F=(G,\theta)$ and a bigraphon $W\function{\Omega\times\Lambda}{\RR_+}$,
we let the function
$t(F,W)\function{\Omega^{V_1(G)\cap\im(\theta)}\times\Lambda^{V_2(G)\cap\im(\theta)}}{\RR_+}$
be given by
\begin{align*}
  t(F,W)(x,y)
  & \df
  \int_{X^{V_1(G)\setminus\im(\theta)}\times Y^{V_2(G)\setminus\im(\theta)}}
  \prod_{(v,w)\in E(G)} W(x''_v,y''_w)
  \ d(\mu\otimes\nu)(x',y'),
\end{align*}
where
\begin{align*}
  x''_v
  & \df
  \begin{dcases*}
    x_v, & if $v\in V_1(G)\cap\im(\theta)$,\\
    x'_v, & if $v\in V_1(G)\setminus\im(\theta)$;
  \end{dcases*}
  &
  y''_w
  & \df
  \begin{dcases*}
    y_w, & if $w\in V_2(G)\cap\im(\theta)$,\\
    y'_w, & if $w\in V_2(G)\setminus\im(\theta)$.
  \end{dcases*}
\end{align*}
When $V_1(G)\cap\im(\theta) = \varnothing$, we will simplify the notation to
$t(F,W)(y)$, and likewise for $V_2(G)\cap\im(\theta) = \varnothing$. We define further
\begin{align*}
  \delta(F,W) & \df \essinf\{t(F,W)(x,y) \mid (x,y)\in X^{V_1(G)\cap\im(\theta)}\times Y^{V_2(G)\cap\im(\theta)}\};\\
  \Delta(F,W) & \df \esssup\{t(F,W)(x,y) \mid (x,y)\in X^{V_1(G)\cap\im(\theta)}\times Y^{V_2(G)\cap\im(\theta)}\}.
\end{align*}
A bigraphon $W$ is called \emph{$F$-regular} if
$\delta(F,W)=\Delta(F,W)=t(\lvert F\rvert,W)$ (of course, equality between any two of
these implies that all of them are equal). For the particular cases of
$e_1$-regular and $e_2$-regular we use the names \emph{left regular} and
\emph{right regular}, respectively. A bigraphon is \emph{biregular} if it is
both left regular and right regular.

A \emph{graphon} is a symmetric bigraphon $W$ in the sense that
$\Omega=\Lambda$ and $W(x,y)=W(y,x)$ for every $x,y\in X$. As mentioned in
the introduction, a \emph{Sidorenko bigraph} $G$ is a bigraph such that
$t(G,W)\geq t(\rho,W)^{e(G)}$ for every bigraphon $W$. A \emph{symmetrically
Sidorenko bigraph} $G$ is a bigraph such that $t(G,W)\geq t(\rho,W)^{e(G)}$
for every graphon $W$ (in this case, one can think of $G$ as of a bipartite
graph since the choice of bipartition does not affect this inequality). Clearly,
every Sidorenko bigraph is also symmetrically Sidorenko but whether the
converse is true is unknown.

\subsection{Weak domination and reflective tree decompositions}

{\em Definitions in this section are needed only for Theorem~\rm\ref{thm:reftree}.}

\smallskip

Inspired by~\cite{CL17}, we give the following definition of weak domination
between bigraphs.

\begin{definition}
  Let $G_1$ and $G_2$ be bigraphs. We say that $G_1$ \emph{weakly dominates} $G_2$ if
  $$
  \frac{t(G_1,W)}{t(\rho,W)^{e(G_1)}}\geq  \frac{t(G_2,W)}{t(\rho,W)^{e(G_2)}}
  $$
  for every {\em biregular} non-zero bigraphon $W$. We say that a
  bigraph $G$ is \emph{induced-Sidorenko} if it weakly dominates all of its induced subgraphs.
\end{definition}

\begin{remark}
In~\cite[\S 5.2]{CL17}, domination between bigraphs $G_1$ and $G_2$ is
defined by the requirement $t(G_1, W)^{1/e(G_1)}\geq t(G_2, W)^{1/e(G_2)}$
for every bigraphon $W$. It is easy to see that as long as $e(G_1)\geq
e(G_2)$ and $G_2$ is Sidorenko (which is the case we are mostly interested
in), domination implies weak domination. That explains our choice of the
terminology. Let us also note that our main result, Theorem~\ref{thm:biregularity},
readily implies that if $G_1$ weakly dominates $G_2$ and $G_2$ is Sidorenko then
$G_1$ is Sidorenko as well.
\end{remark}

Recall that for a bigraph $G$, the \emph{$2$-core} of $G$ is a maximal
connected subgraph in which all vertices have degree at least $2$. When $G$
is connected, it contains only one $2$-core, which we denote $C_2(G)$. It
can be obtained by progressively removing, in an arbitrary order, vertices of degree less than $2$
until no such vertices remain.

For a flag $F=(G,\theta)$ with $G$ connected,
we define the \emph{$2$-core} $C_2(F)$ as the flag of the form
$F'=(G',\theta)$, where $G'$ is the maximal subgraph in which all vertices
that are not in $\im(\theta)$ have degree at least two; this can of course be
obtained by progressively removing vertices of degree less than $2$ that are
not in $\im(\theta)$.

\begin{remark}\label{rmk:2coreweakdom}
  Since in a biregular bigraphon $W$, we have $t(G,W)=t(G-v,W)t(\rho,W)^{d_G(v)}$ whenever $d_G(v)\leq 1$, it
  follows that $G$ weakly dominates $H$ if and only if $C_2(G)$ weakly dominates $C_2(H)$.
\end{remark}

We now define a generalization of the notions of strong tree decompositions~\cite[\S 1]{CKLL18a} and
$N$-decompositions~\cite[\S 5.3]{CL17}, which themselves are generalizations of the usual notion of tree
decompositions~\cite{Hal76,RS84}.

\begin{definition}\label{def:reftree}
  Given a connected non-trivial bigraph $G$, a reflective tree decomposition of $G$ is a tree $T$ such that
  \begin{enumerate}
  \item We have $V(T)\subseteq 2^{V(G)}$ and $V(G) = \bigcup_{U\in V(T)} U$.
    \label{it:partitionvertices}
  \item For every $(v,w)\in E(G)$, there exists $U\in V(T)$ such that $v,w\in U$.
    \label{it:coversedges}
  \item For every $U_1,U_2\in V(T)$ and every $U_3\in V(T)$ in the unique path from $U_1$ to $U_2$ in $T$, we
    have $U_1\cap U_2\subseteq U_3$.
    \label{it:intersectionproperty}
  \item For every $\{U_1,U_2\}\in E(T)$ we have $C_2(F_{U_1 U_2})\cong
      C_2(F_{U_2 U_1})$, where $F_{U_iU_j}\df (G\rest_{U_i},U_1\cap U_2)$
      (we assume that each vertex of $U_1\cap U_2$ receives the same label
      in $F_{U_1U_2}$ as in $F_{U_2U_1}$). \label{it:reflective}
  \end{enumerate}

  Condition~\ref{it:reflective} above in particular implies that for every $U_1,U_2\in V(T)$, we have
  $C_2(G\rest_{U_1})\cong C_2(G\rest_{U_2})$ (since $C_2(\lvert F\rvert)=C_2(\lvert C_2(F)\rvert)$); this
  common $2$-core bigraph is called the \emph{core} of the
  decomposition.
\end{definition}

\begin{remark}\label{rmk:reftree}
  The fact that $G$ is connected implies that each $\lvert F_{U_1 U_2}\rvert$ for $\{U_1,U_2\}\in E(T)$ and each
  $G\rest_U$ for $U\in V(T)$ is connected. Furthermore, condition~\ref{it:reflective} is equivalent to the
  same condition obtained by replacing $F_{U_iU_j}$ with $F_{U_i U_j}'\df (G\rest_{U_i}-E(G\rest_{U_1\cap
    U_2}), U_1\cap U_2)$ and it also equivalent to the existence of an automorphism of the flag
  $F\df(C_2(G\rest_{U_1\cup U_2}), U_1\cap U_2)$ that maps $U_1\cap V(\lvert F\rvert)$ to $U_2\cap V(\lvert F\rvert)$.

  Items~\ref{it:partitionvertices}, \ref{it:coversedges} and~\ref{it:intersectionproperty} alone
  say that $T$ is a usual tree decomposition. Strong tree decompositions
  of~\cite[\S 1]{CKLL18a} are precisely reflective tree decompositions whose core is empty (i.e., $G\rest_U$ is
  a tree for every $U\in V(T)$) and $N$-decompositions of~\cite[\S 5.3]{CL17} are obtained by replacing the
  requirement $C_2(F_{U_1U_2})\cong C_2(F_{U_2U_1})$ in~\ref{it:reflective} with $F_{U_1U_2}\cong F_{U_2U_1}$
  instead (this forces all $G\rest_U$ for $U\in V(T)$ to be isomorphic to a fixed bigraph $N$).
\end{remark}

\section{Main results}
\label{sec:results}

In this section we present our main results.

\begin{theorem}\label{thm:biregularity}
  Let $G$ be a bigraph. If there exists $c_G > 0$ such that $t(G,W)\geq c_G\cdot t(\rho,W)^{e(G)}$ for every
  biregular bigraphon $W$, then $G$ is a Sidorenko bigraph.
\end{theorem}

\begin{theorem}\label{thm:amalgamation}
  If $F_1$ and $F_2$ are left {\rm (}or right{\rm )} $1$-flags such that $\lvert F_1\rvert$ and $\lvert F_2\lvert$ are Sidorenko bigraphs, then
  $\lvert F_1\sqcup F_2\rvert$ is a Sidorenko bigraph.
\end{theorem}

The next theorem can be seen as a partial converse to Theorem~\ref{thm:amalgamation}.

\begin{theorem}\label{thm:power}
  Let $F$ be a left $1$-flag and $k\in\NN_+$. Then $\lvert F\rvert$ is a Sidorenko bigraph if and only if $\lvert F^{\sqcup k}\rvert $ is a Sidorenko bigraph.
\end{theorem}

As we mentioned in the introduction, Theorems~\ref{thm:amalgamation} and~\ref{thm:power} also follow from~\cite{Sze15b}
(but our proofs are simpler).

\begin{theorem}\label{thm:stars}
  If $G$ is a Sidorenko bigraph with $\delta_1(G)\geq d$, then $t(G,W)\geq t(K_{1,d},W)^{e(G)/d}$ for every
  bigraphon $W$.
\end{theorem}

\begin{theorem}\label{thm:reftree}
  If $T$ is a reflective tree decomposition of a connected non-trivial bigraph $G$ whose core $H$ weakly
  dominates $G\rest_{U_1\cap U_2}$ for every $\{U_1,U_2\}\in E(T)$, then $G$ weakly dominates $H$. In
  particular, if $H$ is a Sidorenko bigraph, then $G$ is also a Sidorenko bigraph.
\end{theorem}

Note that in Theorem~\ref{thm:reftree} above, if the core $H$ is an induced-Sidorenko bigraph, then both the
condition that it weakly dominates $G\rest_{U_1\cap U_2}$ for every $\{U_1,U_2\}\in E(T)$
and the fact that
$H$ is a Sidorenko bigraph follow (see Remark~\ref{rmk:2coreweakdom}). Hence in that case we can conclude that
$G$ is a Sidorenko bigraph.

As we mentioned in Remark~\ref{rmk:reftree}, the notions of strong tree
decompositions and $N$-decompositions are particular cases of reflective tree
decompositions. The corresponding results can be retrieved from
Theorem~\ref{thm:reftree} above as follows. For~\cite[Theorem~1.2]{CKLL18a},
any two forests weakly dominate each other for obvious reasons, which implies
that strongly tree decomposable bigraphs are Sidorenko bigraphs.
For~\cite[Theorem~5.12]{CL17}, by~\cite[Theorem~2.14]{Hat10}, every weakly
norming bigraph $N$ dominates any of its (not necessarily induced) subgraphs,
so it is an induced-Sidorenko bigraph, hence any $N$-decomposable bigraph for
a weakly norming bigraph $N$ is a Sidorenko bigraph.

However, let us note that there are many induced-Sidorenko bigraphs that are
not weakly norming bigraphs. For example, any weakly norming bigraph without isolated vertices is
necessarily biregular~\cite[Theorem~2.10(ii)]{Hat10}, but the
induced-Sidorenko property is trivially preserved under amalgamations with
trees along a single vertex, which will destroy biregularity. For a less
trivial example, let $B_k$ be the \emph{$k$-book bigraph} (see
Figure~\ref{fig:book}), that is, the graph obtained by gluing $k$ copies of
$4$-cycles along the same edge; since we can also see $B_k$ as the
amalgamation of two copies of $B_{k-1}$ along a $B_{k-2}$ (with the
convention $B_0\df\rho$), by inductive application of
Theorem~\ref{thm:reftree} above and Remark~\ref{rmk:2coreweakdom}, all $B_k$
are induced-Sidorenko bigraphs.

\begin{figure}[htb]
  \begin{center}
    \begingroup

\def\pointsize{1 pt}
\def\vertsep{1}
\def\horzsep{1}
\def\angle{30}
\def\diststep{0.5}
\def\vertsize{1.5cm}
\def\uppervertsize{0.75cm}

\newcommand{\bookfig}[1]{%
  \begin{tikzpicture}
    \path[use as bounding box] (0,\uppervertsize) rectangle (0,-\vertsize);
    
    \coordinate (P) at (0,0);
    \coordinate (Q) at (0,-\vertsep);

    \pgfmathsetmacro{\secondangle}{180-\angle}

    \foreach \i in {0,...,#1}{%
      \pgfmathsetmacro{\dist}{\i * \diststep}
      \coordinate (A\i) at ($(P) + (\horzsep,0) + (\angle:\dist)$);
      \coordinate (B\i) at ($(Q) + (\horzsep,0) + (-\angle:\dist)$);
      \coordinate (C\i) at ($(P) + (-\horzsep,0) + (\secondangle:\dist)$);
      \coordinate (D\i) at ($(Q) + (-\horzsep,0) + (-\secondangle:\dist)$);
    }

    \filldraw (P) circle (\pointsize);
    \filldraw (Q) circle (\pointsize);

    \foreach \i in {0,...,#1}{%
      \filldraw (A\i) circle (\pointsize);
      \filldraw (B\i) circle (\pointsize);
      \filldraw (C\i) circle (\pointsize);
      \filldraw (D\i) circle (\pointsize);

      \draw (P) -- (A\i) -- (B\i) -- (Q) -- cycle;
      \draw (P) -- (C\i) -- (D\i) -- (Q) -- cycle;
    }
  \end{tikzpicture}
}

\begin{subfigure}[b]{0.3\textwidth}
  \begin{center}
    \bookfig{0}
    \caption*{$B_2$}
  \end{center}
\end{subfigure}
\quad
\begin{subfigure}[b]{0.3\textwidth}
  \begin{center}
    \bookfig{1}
    \caption*{$B_4$}
  \end{center}
\end{subfigure}
\quad
\begin{subfigure}[b]{0.3\textwidth}
  \begin{center}
    \bookfig{2}
    \caption*{$B_6$}
  \end{center}
\end{subfigure}

\endgroup
%% Local Variables:
%% mode: latex
%% End:
    \caption{Book bigraphs.}
    \label{fig:book}
  \end{center}
\end{figure}
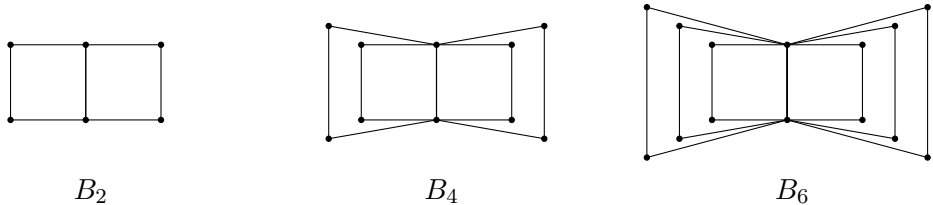

\section{The tensor power trick and the main lemma}
\label{sec:mainlemma}

We start with a slightly more general version of the tensor power trick
of~\cite[Remark~2]{Sid91} that we will need later.

\begin{lemma}[Tensor power trick]\label{lem:tensor}
  Let $\cW$ be a class of bigraphons that is closed under tensor powers, let $G_1,\ldots,G_n,H_1,\ldots,H_m$
  be bigraphs and $r_1,\ldots,r_n,s_1,\ldots,s_m\in\RR_+$. If there exists $c > 0$ such that
  \begin{align*}
    \prod_{i=1}^n t(G_i,W)^{r_i} & \geq c\cdot\prod_{j=1}^m t(H_j,W)^{s_j}
  \end{align*}
  for every $W\in\cW$, then the same inequality holds with $c$ replaced by $1$.
\end{lemma}

\begin{proof}
  Since $\cW$ is closed under tensor powers, for $W\in\cW$ and $k\in\NN_+$, we have
  \begin{align*}
    \prod_{i=1}^n t(G_i,W)^{r_i}
    & =
    \left(\prod_{i=1}^n t(G_i,W^{\otimes k})^{r_i}\right)^{1/k}
    \\
    & \geq
    \left(c\cdot \prod_{j=1}^m t(H_j,W^{\otimes k})^{s_j}\right)^{1/k}
    =
    c^{1/k}\cdot\prod_{j=1}^m t(H_j,W)^{s_j}
  \end{align*}
  and letting $k\to\infty$ gives the result.
\end{proof}

The proof of the biregularity result, Theorem~\ref{thm:biregularity}, consists of the construction of a
biregular bigraphon $W'$ from a bigraphon $W$ with the following properties:

\begin{enumerate}
\item \label{it:rho_preserved} $t(\rho,W')=t(\rho,W)$;

\item \label{it:G_bounded} for every bigraph $G$ there exists a constant $C_G>0$ depending only on $G$ such that for any bigraphon
$W$, $t(G,W')\leq C_G\cdot t(G,W)$.
\end{enumerate}

This construction will actually be a chain of
steps so that we obtain progressively better ``regularity properties''. Namely, the steps are:
\begin{enumerate}[label={\arabic*.}, ref={(\arabic*)}]
\item $W_1$ satisfies $\Delta(e_1,W_1)\leq 2\cdot t(\rho,W_1)$.
\item $W_2$ satisfies $\max\{\Delta(e_1,W_2),\Delta(e_2,W_2)\}\leq 2\cdot t(\rho,W_2)$.
\item $W_3$ satisfies $\min\{\delta(e_1,W_3),\delta(e_2,W_3)\}\geq 2^{-10}\cdot t(\rho,W_3)$.
  \label{it:lowerbound}
\item $W_4$ is left regular and satisfies $\delta(e_2,W_4)\geq 2^{-10}\cdot t(\rho,W_4)$.
\item $W_5$ is biregular.
\end{enumerate}

It turns out that all steps except for step~\ref{it:lowerbound} can be performed by the same construction in
Lemma~\ref{lem:main} below that improves the ``quality of regularity'' of its input. In fact, we will state this construction in a slightly more general setting so
that we can also use it for Theorems~\ref{thm:power} and~\ref{thm:stars} (for Theorem~\ref{thm:biregularity}
we will take $F = e_1$, for Theorem~\ref{thm:power} we will take $F$ as in its statement and for
Theorem~\ref{thm:stars}, we will take $F = K_{1,d}^L$).

\begin{lemma}\label{lem:main}
  Let $d\in\NN_+$, let $F=(G,\theta)$ be a left $1$-flag such that $G$ is left $d$-regular and let $\epsilon >
  0$.

  Then for every bigraphon $W\function{\Omega\times\Lambda}{\RR_+}$ over spaces $\Omega=(X,\mu)$ and
  $\Lambda=(Y,\nu)$, there exists a bigraphon $W'\function{\Omega'\times\Lambda}{\RR_+}$ such that the
  following hold.
  \begin{enumerate}
  \item We have $\Delta(F,W')\leq (1+\epsilon)\cdot t(G,W')$.
    \label{lem:main:DeltaF}
  \item We have
    \begin{align*}
      \delta(F,W') & \geq \min\left\{t(G,W'), \frac{\delta(F,W)}{\epsilon}\right\}.
    \end{align*}
    \label{lem:main:deltaF}
  \item For every right $1$-flag $F'=(G',\theta')$ such that $G'$ is left $d$-regular, we have
    \begin{align*}
      t(F',W')(y) & = t(F',W)(y)
    \end{align*}
    for every $y\in Y$.
    \label{lem:main:right1flag}
  \item For every bigraph $G'$ with $\Delta_1(G')\leq d$, we have
    \begin{align*}
      t(G',W') & \geq \left(1 + \frac{1}{\epsilon}\right)^{e(G')/d - v_1(G')}\cdot t(G',W).
    \end{align*}
    \label{lem:main:tG'lower}
  \item For every bigraph $G'$ with $\delta_1(G')\geq d$, we have
    \begin{align*}
      t(G',W') & \leq \left(1 + \frac{1}{\epsilon}\right)^{e(G')/d - v_1(G')}\cdot t(G',W).
    \end{align*}
    \label{lem:main:tG'upper}
  \item For every left $d$-regular bigraph $G'$, we have $t(G',W')=t(G',W)$.
    \label{lem:main:tG'}
  \end{enumerate}
\end{lemma}

\begin{proof}
  If $t(G,W) = 0$, then we can simply take $\Omega'\df\Omega$ and $W'\df W$, so suppose that $t(G,W) > 0$.

  Define the function $f\function{X}{\RR_+}$ by
  \begin{align}\label{eq:fx}
    f(x) & \df \max\{t(F,W)(x), \epsilon\cdot t(G,W)\} \geq \epsilon\cdot t(G,W) > 0
  \end{align}
  and let $Z\df\int_X f(x)\ d\mu(x)$. Let $\Omega'\df(X,\mu')$, where $\mu'$ is the probability measure such
  that $d\mu'(x) = (f(x)/Z)\ d\mu(x)$. Since $t(F,W)(x)\leq f(x)\leq t(F,W)(x) + \epsilon\cdot t(G,W)$, we have
  \begin{align}\label{eq:Zbounds}
    0 < t(G,W) & \leq Z \leq (1+\epsilon)\cdot t(G,W).
  \end{align}

  Define now the bigraphon $W'\function{\Omega'\times\Lambda}{\RR_+}$ by
  \begin{align*}
    W'(x,y) & \df \left(\frac{Z}{f(x)}\right)^{1/d}\cdot W(x,y).
  \end{align*}
  (Note that~\eqref{eq:fx} and the upper bound of~\eqref{eq:Zbounds} imply that
  $W'$ is bounded.)

  We start by showing that $W'$ satisfies the last three items. Indeed, if $G'$ is a bigraph, then we have
  \begin{equation}\label{eq:tG'}
    \begin{aligned}
    t(G',W')
    & =
    \int_{X^{V_1(G')}\times Y^{V_2(G')}} \prod_{(v,w)\in E(G')} W'(x_v,y_w)\ d(\mu'\otimes\nu)(x',y)
    \\
    & =
    \begin{multlined}[t]
      Z^{e(G')/d - v_1(G')}\int_{X^{V_1(G')}\times Y^{V_2(G')}} \prod_{(v,w)\in E(G')} W(x_v,y_w)
      \\
      \cdot\prod_{v\in V_1(G')} f(x_v)^{1 - d_{G'}(v)/d}
      \ d(\mu\otimes\nu)(x,y).
    \end{multlined}
    \end{aligned}
  \end{equation}

  If $\Delta_1(G')\leq d$, then the exponent of $Z$ in the above is non-positive and the exponent of $f(x_v)$ is
  non-negative, hence~\eqref{eq:fx} and the upper bound of $Z$ in~\eqref{eq:Zbounds} imply
  \begin{align*}
    t(G',W')
    & \geq
    ((1+\epsilon)\cdot t(G,W))^{e(G')/d - v_1(G')}\cdot t(G',W)
    \cdot \prod_{v\in V_1(G')} (\epsilon\cdot t(G,W))^{1 - d_{G'}(v)/d}
    \\
    & =
    \left(1 + \frac{1}{\epsilon}\right)^{e(G')/d - v_1(G')} t(G',W).
  \end{align*}
  Thus, item~\ref{lem:main:tG'lower} follows.

  On the other hand, if $\delta_1(G')\geq d$ instead, then the exponent of $Z$ is non-negative and the
  exponent of $f(x_v)$ is non-positive, so the same bounds on $f(x)$ and $Z$ flip the inequality above to give
  item~\ref{lem:main:tG'upper}. Item~\ref{lem:main:tG'} follows by combining items~\ref{lem:main:tG'lower}
  and~\ref{lem:main:tG'upper} when $G'$ is left $d$-regular.

  \medskip

  For items~\ref{lem:main:DeltaF} and~\ref{lem:main:deltaF}, since $G$ is left $d$-regular, a calculation
  analogous to the one in~\eqref{eq:tG'} for $t(F,W')(x)$ has the exponent of $Z$ being $1$ (since the
  labeled vertex is not integrated out) and exponents of all $f(x_v)$ being $0$ except for the one
  corresponding to the labeled vertex, which has exponent $-1$ instead (as the labeled vertex is not
  integrated), so we get
  \begin{align*}
    t(F,W')(x) & = \frac{Z}{f(x)}\cdot t(F,W)(x) \leq Z \leq (1+\epsilon)\cdot t(G,W),
  \end{align*}
  where the inequalities follow from~\eqref{eq:fx} and the upper bound in~\eqref{eq:Zbounds},
  respectively. Thus, item~\ref{lem:main:DeltaF} follows.

  On the other hand, by using the full definition of $f(x)$ from~\eqref{eq:fx} and the lower bound
  in~\eqref{eq:Zbounds} instead, we get
  \begin{align*}
    t(F,W')(x) & \geq \min\left\{t(G,W), \frac{t(F,W)(x)}{\epsilon}\right\},
  \end{align*}
  and since $t(G,W)=t(G,W')$ by item~\ref{lem:main:tG'}, we conclude that item~\ref{lem:main:deltaF} holds.

  \medskip

  Finally, for item~\ref{lem:main:right1flag}, since $G'$ is left $d$-regular, a calculation analogous to the
  one in~\eqref{eq:tG'} has exponents of $Z$ and the $f(x_v)$ all zero (as the labeled vertex is on the
  right), so we get $t(F',W')(y) = t(F',W)(y)$.
\end{proof}

\section{Biregularity}
\label{sec:biregularity}

In this section we prove our biregularity result, Theorem~\ref{thm:biregularity}. Let us first extract from
Lemma~\ref{lem:main} its partial case $d=1$, $F=e_1$, $F'= e_2$ needed for that purpose.

\begin{corollary}\label{cor:main}
  For every $\epsilon > 0$ and every bigraphon $W\function{\Omega\times\Lambda}{\RR_+}$ over spaces
  $\Omega=(X,\mu)$ and $\Lambda=(Y,\nu)$, there exists a bigraphon $W'\function{\Omega'\times\Lambda}{\RR_+}$
  such that the following hold.
  \begin{enumerate}
  \item We have $\Delta(e_1,W')\leq (1+\epsilon)\cdot t(\rho,W')$.
    \label{cor:main:DeltaF}
  \item We have
    \begin{align*}
      \delta(e_1,W') & \geq \min\left\{t(\rho,W'), \frac{\delta(e_1,W)}{\epsilon}\right\}.
    \end{align*}
    \label{cor:main:deltaF}
  \item For every $y\in Y$, $t(e_2,W')(y) = t(e_2,W)(y)$. Therefore, $t(\rho,W')=t(\rho,W)$.
    \label{cor:main:right1flag}
  \item For every bigraph $G$ we have
    \begin{align*}
      t(G,W') & \leq \left(1 + \frac{1}{\epsilon}\right)^{e(G)}\cdot t(G,W).
    \end{align*}
    \label{cor:main:tG'upper}
  \end{enumerate}
\end{corollary}

\begin{proof}
  The only thing to be explained here is the absence of the condition $\delta_1(G)\geq 1$ in
  item~\ref{cor:main:tG'upper} (that corresponds to item~\ref{lem:main:tG'upper} in Lemma~\ref{lem:main}). This is simply because removing all isolated vertices in $V_1(G)$ does not
  change any of the three quantities in this inequality.
\end{proof}

As we will see below, Corollary~\ref{cor:main} will take care of all steps
in our program, except for $W_2\Longrightarrow W_3$. This remaining step is performed by Lemma~\ref{lem:lowerreg},
which can be seen as a limit, non-symmetric version of the argument in~\cite[Lemma~3.4]{CKLL18a}.

\begin{lemma}\label{lem:lowerreg}
  If $W\function{\Omega\times\Lambda}{\RR_+}$ is a bigraphon over spaces $\Omega=(X,\mu)$ and
  $\Lambda=(Y,\nu)$ such that
  \begin{equation} \label{eq:assumption}
  \max\{\Delta(e_1,W),\Delta(e_2,W)\}\leq 2\cdot t(\rho,W),
  \end{equation}
  then there exists a
  bigraphon $W'\function{\Omega'\times\Lambda'}{\RR_+}$ such that the following hold.
  \begin{enumerate}
   \item We have $t(\rho,W') = t(\rho,W)$.
    \label{lem:lowerreg:trho}
  \item We have $\min\{\delta(e_1,W'),\delta(e_2,W')\}\geq 2^{-10}\cdot t(\rho,W')$.
    \label{lem:lowerreg:deltaej}
  \item For every bigraph $G$, we have $t(G,W')\leq 2^{3v(G) + e(G)}\cdot t(G,W)$.
    \label{lem:lowerreg:tG}
  \end{enumerate}
\end{lemma}

\begin{proof}
  If $t(\rho,W) = 0$, then we can simply take $\Omega'\df\Omega$, $\Lambda'\df\Lambda$ and $W'\df W$, so
  suppose $t(\rho,W) > 0$.

  Without loss of generality, let us assume that the probability spaces $\Omega$ and $\Lambda$ are atomless
  (if not, we can simply replace each atom of the space by a copy of an interval of appropriate length
  equipped with Lebesgue measure).

  Let $\alpha\in (0,1)$, to be specified later. We define sequences of measurable sets $X = X_0\supseteq
  X_1\supseteq X_2\supseteq\cdots$ and $Y = Y_0\supseteq Y_1\supseteq Y_2\supseteq\cdots$ using the following
  algorithm.
  \begin{enumerate}[label={\arabic*.}, ref={\arabic*}]
  \item Let $X_0\df X$ and $Y_0\df Y$.
  \item Given $X_i$ and $Y_i$,
    \begin{enumerate}[label={\alph*.}, ref={\alph*}]
    \item if there exists a measurable $R^1_i\subseteq X_i$ with $\mu\rest_{X_i}(R^1_i) = \alpha$ and
      $t(e_1,W\rest_{X_i\times Y_i})(x) < t(\rho,W\rest_{X_i\times Y_i})/10$ for every $x\in R^1_i$, then
      let $X_{i+1}\df X_i\setminus R^1_i$ and $Y_{i+1}\df Y_i$;
    \item otherwise, if there exists a measurable $R^2_i\subseteq Y_i$ with $\nu\rest_{Y_i}(R^2_i) = \alpha$ and
      $t(e_2,W\rest_{X_i\times Y_i})(y) < t(\rho,W\rest_{X_i\times Y_i})/10$ for every $y\in R^2_i$, then
      let $X_{i+1}\df X_i$ and $Y_{i+1}\df Y_i\setminus R^2_i$;
    \item otherwise, stop the construction.
    \end{enumerate}
  \end{enumerate}

  The first order of business is to show that the construction above stops in finitely many steps. To do so,
  note first that if $X_i$ gets changed at some stage $i$, then we have
  \begin{align*}
    t(\rho,W\rest_{X_{i+1}\times Y_{i+1}})
    & =
    \frac{
      t(\rho,W\rest_{X_i\times Y_i}) - \int_{R^1_i} t(e_1,W\rest_{X_i\times Y_i})(x)\ d\mu\rest_{X_i}(x)
    }{
      \mu\rest_{X_i}(X_{i+1})
    }
    \\
    & \geq
    \frac{1 - \alpha/10}{1 - \alpha}\cdot t(\rho,W\rest_{X_i\times Y_i}).
  \end{align*}
  By symmetry, the same conclusion holds when $Y_i$ gets changed.
  Thus, by induction, we conclude that whenever the algorithm proceeds to the
  $i$th stage, we have
  \begin{align}\label{eq:exponentialedgedensity}
    t(\rho,W\rest_{X_i\times Y_i})
    & \geq
    \left(\frac{1- \alpha/10}{1 - \alpha}\right)^i t(\rho,W)
    \geq
    t(\rho,W).
  \end{align}
  Since $t(\rho,W\rest_{X_i\times Y_i}) \leq \lVert W\rVert_\infty$, it follows that the construction
  indeed must halt in finitely many steps. Let $i_0$ be the step at which this happens.
  The heart of the whole argument is to bound $i_0$ as {\em a
  function of $\alpha$ only}; more specifically, we are going to prove that
  \begin{align}\label{eq:maxsteps}
    i_0 & \leq \frac{1}{\log_2(1 - \alpha/10) - \log_2\sqrt{1-\alpha}}.
  \end{align}

  Towards that end, let us define
  \begin{align*}
    D_i
    & \df
    \Delta(e_1,W\rest_{X_i\times Y_i})\cdot\Delta(e_2,W\rest_{X_i\times Y_i}).
  \end{align*}
  We claim that for every $i\in\{0,\ldots,i_0\}$, we have
  $$
    D_i\leq \frac{4t(\rho,W)^2}{(1-\alpha)^i};
  $$
we prove this by induction on $i$.

For $i=0$ this immediately follows from the assumption~\eqref{eq:assumption}.

For the inductive step, if the bigraphon gets decreased by removing $R_i^j\
(j=1,2)$ then we have
  \begin{align*}
    \Delta(e_j,W\rest_{X_{i+1}\times Y_{i+1}}) & \leq \Delta(e_j,W\rest_{X_i\times Y_i});\\
    \Delta(e_{3-j},W\rest_{X_{i+1}\times Y_{i+1}}) & \leq \frac{\Delta(e_{3-j},W\rest_{X_i\times Y_i})}{1-\alpha}.
  \end{align*}
  This completes the inductive step.

\medskip
  We now conclude that
  \begin{align*}
    t(\rho,W\rest_{X_{i_0}\times Y_{i_0}})
    & \leq
    \min\{\Delta(e_1,W\rest_{X_{i_0}\times Y_{i_0}}), \Delta(e_2,W\rest_{X_{i_0}\times Y_{i_0}})\}
    \\
    & \leq \sqrt{D_{i_0}} \leq
    \frac{2}{(1-\alpha)^{i_0/2}}\cdot t(\rho,W).
  \end{align*}
Comparing this with the bound~\eqref{eq:exponentialedgedensity} gives us~\eqref{eq:maxsteps},
as desired.

  Note also that a simple induction gives
  \begin{align}\label{eq:measures}
    \min\{\mu(X_{i_0}),\nu(Y_{i_0})\} & \geq (1 - \alpha)^{i_0}.
  \end{align}

  Let
  \begin{align*}
    X'
    & \df
    \{x\in X_{i_0} \mid t(e_1,W\rest_{X_{i_0}\times Y_{i_0}})(x)\geq t(\rho,W\rest_{X_{i_0}\times Y_{i_0}})/10\},
    \\
    Y'
    & \df
    \{y\in Y_{i_0} \mid t(e_2,W\rest_{X_{i_0}\times Y_{i_0}})(y)\geq t(\rho,W\rest_{X_{i_0}\times Y_{i_0}})/10\}
  \end{align*}
  and note that since the probability spaces $\Omega$ and $\Lambda$ are atomless, we must have
  $\mu\rest_{X_{i_0}}(X')\geq 1-\alpha$ and $\nu\rest_{Y_{i_0}}(Y')\geq 1-\alpha$ since otherwise
  we could have continued with the algorithm. Together
  with~\eqref{eq:maxsteps} and~\eqref{eq:measures}, this gives
  \begin{align*}
    \min\{\mu(X'),\nu(Y')\}
    & \geq
    (1-\alpha)^{1 + 1/(\log_2(1-\alpha/10) - \log_2\sqrt{1-\alpha})}.
  \end{align*}
  Let $M=M(\alpha)$ be the right-hand side of the above; a straightforward
  calculation shows that
  \begin{equation} \label{eq:limit}
  \lim_{\alpha\to 0} M(\alpha) =\frac{\sqrt 2}{8}.
  \end{equation}

  Now, let $\widehat{W} \df W\rest_{X'\times Y'}$.
Note that for every bigraph $G$, we have
  \begin{align}\label{eq:tGrestricted}
    t(G,\widehat{W})
    & \leq
    \frac{t(G,W)}{\mu(X')^{v_1(G)}\cdot\nu(Y')^{v_2(G)}}
    \leq
    \frac{t(G,W)}{M^{v(G)}}.
  \end{align}
Note also that by~\eqref{eq:exponentialedgedensity},
  \begin{equation}\label{eq:trhorestricted}
    \begin{aligned}
      t(\rho,\widehat{W})
      & \geq
      \begin{multlined}[t]
        t(\rho,W\rest_{X_{i_0}\times Y_{i_0}})
        \\
        - \int_{X_{i_0}\setminus X'} t(e_1,W\rest_{X_{i_0}\times Y_{i_0}})(x)\ d\mu\rest_{X_{i_0}}(x)
        \\
        - \int_{Y_{i_0}\setminus Y'} t(e_2,W\rest_{X_{i_0}\times Y_{i_0}})(y)\ d\nu\rest_{Y_{i_0}}(y)
      \end{multlined}
      \\
      & \geq
      \left(1 - \frac{\alpha}{5}\right)\cdot t(\rho,W\rest_{X_{i_0}\times Y_{i_0}})
      \geq
      \frac{t(\rho,W)}{2}.
    \end{aligned}
  \end{equation}
Finally, from the definition of $X'$ and $Y'$, we have
  \begin{equation}\label{eq:deltaej}
    \begin{aligned}
      \min\{\delta(e_1,\widehat{W}),\delta(e_2,\widehat{W})\}
      & \geq
      \frac{t(\rho,W\rest_{X_{i_0}\times Y_{i_0}})}{10} - \alpha\cdot\lVert W\rVert_\infty
      \\
      & \geq
      \frac{t(\rho,W)}{10} - \alpha\cdot\lVert W\rVert_\infty,
    \end{aligned}
  \end{equation}
  where the first inequality follows from $\mu(X_{i_0}\setminus X'),\nu(Y_{i_0}\setminus Y')<\alpha$ and the
  second inequality again follows from~\eqref{eq:exponentialedgedensity}.

By~\eqref{eq:limit}, if we choose $\alpha\in (0,1)$ small enough then
  \begin{align}\label{eq:conditions}
    M(\alpha) & \geq \frac{1}{8}, &
    \frac{1}{10} - \alpha\cdot\frac{\lVert W\rVert_\infty}{t(\rho,W)} & \geq \frac{1}{16}.
  \end{align}
Finally, define
  \begin{align*}
    W' & \df \frac{t(\rho,W)}{t(\rho,\widehat{W})}\cdot \widehat{W}
  \end{align*}
  so that item~\ref{lem:lowerreg:trho} follows trivially. Since $t(\rho,\widehat{W})\geq t(\rho,W)/2$, the
  first condition in~\eqref{eq:conditions} along with~\eqref{eq:tGrestricted} gives
  \begin{align}\label{eq:tGmultiple}
    t(G,W') & \leq 2^{e(G)}\cdot t(G,\widehat{W}) \leq 2^{3v(G) + e(G)}\cdot t(G,W),
  \end{align}
  and item~\ref{lem:lowerreg:tG} also follows.

  For item~\ref{lem:lowerreg:deltaej}, note that~\eqref{eq:deltaej} and the second condition
  in~\eqref{eq:conditions} imply
  \begin{align*}
    \min\{\delta(e_1,W'),\delta(e_2,W')\}
    & =
    \frac{t(\rho,W)}{t(\rho,\widehat{W})}
    \cdot
    \min\{\delta(e_1,\widehat{W}),\delta(e_2,\widehat{W})\}
    \\
    & \geq
    \frac{t(\rho,W)}{t(\rho,\widehat{W})}
    \cdot
    \frac{t(\rho,W)}{16}
    \\
    & \geq
    \frac{t(\rho,W)}{t(\rho,\widehat{W})}
    \cdot
    2^{-10}\cdot t(\rho,\widehat{W})
    \\
    & =
    2^{-10}\cdot t(\rho,W'),
  \end{align*}
  where the last inequality follows from~\eqref{eq:tGmultiple}.
\end{proof}

\begin{proofof}{Theorem~\ref{thm:biregularity}}
  We make the constructions $W\Longrightarrow W_1\Longrightarrow W_2\Longrightarrow W_3\Longrightarrow
  W_4\Longrightarrow W_5$, where the first two arrows are applications of Corollary~\ref{cor:main} and its
  dual, respectively, with $\epsilon = 1$, the third arrow is an application of Lemma~\ref{lem:lowerreg} and
  the last two arrows are applications of Corollary~\ref{cor:main} and its dual, respectively, with $\epsilon
  = 2^{-10}$.

  Checking all necessary conditions is straightforward, the only thing worth noticing
  is (bi)regularity of $W_4$ and $W_5$. It is implied by the following computation on
  the base of
  items~\ref{cor:main:deltaF} and~\ref{cor:main:right1flag} in Corollary~\ref{cor:main}:
  \begin{align*}
    \delta(e_1,W_5)
    & =
    \delta(e_1,W_4)
    \geq
    \min\left\{t(\rho,W_4), \frac{\delta(e_1,W_3)}{2^{-10}}\right\}
    =
    t(\rho,W_5),
    \\
    \delta(e_2,W_5)
    & \geq
    \min\left\{t(\rho,W_5), \frac{\delta(e_2,W_4)}{2^{-10}}\right\}
    \geq
    \min\left\{t(\rho,W_5), \frac{\delta(e_2,W_3)}{2^{-10}}\right\}
    =
    t(\rho,W_5).
  \end{align*}
  We also have the following chain of inequalities.
  \begin{align*}
    t(G,W)
    & \geq
    2^{-e(G)}\cdot t(G,W_1)
    \geq
    2^{-2e(G)}\cdot t(G,W_2)
    \\
    & \geq
    2^{-3v(G) - 3e(G)}\cdot t(G,W_3)
    \geq
    1025^{-e(G)}\cdot 2^{-3v(G) - 3e(G)}\cdot t(G,W_4)
    \\
    & \geq
    1025^{-2e(G)}\cdot 2^{-3v(G) - 3e(G)}\cdot t(G,W_5)
    \\
    & \geq
    1025^{-2e(G)}\cdot 2^{-3v(G) - 3e(G)}\cdot c_G\cdot t(\rho,W_5)^{e(G)}
    \\
    & =
    1025^{-2e(G)}\cdot 2^{-3v(G) - 3e(G)}\cdot c_G\cdot t(\rho,W)^{e(G)}.
  \end{align*}
  Therefore, $G$ is a Sidorenko bigraph by Lemma~\ref{lem:tensor}.
\end{proofof}

Theorem~\ref{thm:biregularity} has the following simple but very useful corollary
(that of course can be extracted already from the approximate version in~\cite{CKLL18a}).

\begin{corollary}\label{cor:deg1}
  If $v$ is a vertex of degree $1$ in a bigraph $G$, then $G$ is a Sidorenko bigraph if and only if $G-v$ is a
  Sidorenko bigraph.
\end{corollary}

\begin{proof}
  Follows from Theorem~\ref{thm:biregularity} and the fact that in a biregular bigraphon $W$ we have $t(G,W) =
  t(G-v,W)\cdot t(\rho,W)$.
\end{proof}

\section{$1$-flags and $d$-stars}
\label{sec:applications}

In this section we show how Theorem~\ref{thm:biregularity}, along with
Lemma~\ref{lem:main}, yields Theorems~\ref{thm:amalgamation}, \ref{thm:power}
and~\ref{thm:stars}.

\begin{proofof}{Theorem~\ref{thm:amalgamation}}
  Let $W\function{\Omega\times\Lambda}{\RR_+}$ be a biregular bigraphon over spaces $\Omega=(X,\mu)$ and
  $\Lambda=(Y,\nu)$ and for each $i\in [2]$, let $f_i\in\RR_+$ be such that
  \begin{align*}
    \mu(\{x\in X \mid t(F_i,W)(x) < f_i\}) & \leq \frac{1}{3},\\
    \mu(\{x\in X \mid t(F_i,W)(x) \leq f_i\}) & \geq \frac{1}{3}.
  \end{align*}
  Let $U_i\df\{x\in X\mid t(F_i,W)(x)\leq f_i\}$ (so that $\mu(U_i)\geq 1/3$) and let $W_i\df W\rest_{U_i\times Y}$.

  Since $W$ is left regular, it follows that $W_i$ is also left regular and
  hence $t(\rho,W) = t(\rho,W_i)$. On the
  other hand, since $\lvert F_i\rvert$ is a Sidorenko bigraph, we have
  \begin{align*}
    t(\rho,W)^{e(\lvert F_i\rvert)}
    & =
    t(\rho,W_i)^{e(\lvert F_i\rvert)}
    \leq
    t(\lvert F_i\rvert,W_i)
    \\
    & \leq
    \frac{1}{\mu(U_i)^{v_1(\lvert F_i\rvert)}}\int_{U_i} t(F_i,W)(x)\ d\mu(x)
    \leq
    \frac{f_i}{\mu(U_i)^{v_1(\lvert F_i\rvert)-1}}
    \\
    & \leq
    3^{v_1(\lvert F_i\rvert) - 1}\cdot f_i.
  \end{align*}

  Let now $X'\df\{x\in X \mid t(F_1,W)(x)\geq f_1\land t(F_2,W)(x)\geq f_2\}$ and note that $\mu(X')\geq 1/3$,
  so we have
  \begin{align*}
    t(\lvert F_1\sqcup F_2\rvert,W)
    & \geq
    \int_{X'} t(F_1,W)(x)\cdot t(F_2,W)(x)\ d\mu(x)
    \geq
    \frac{1}{3}\cdot f_1\cdot f_2
    \\
    & \geq
    \frac{1}{3^{v_1(\lvert F_1\rvert) + v_1(\lvert F_2\rvert) - 1}}\cdot t(\rho,W)^{e(\lvert F_1\rvert) + e(\lvert F_2\rvert)}
    \\
    & =
    \frac{1}{3^{v_1(\lvert F_1\sqcup F_2\rvert)}}\cdot t(\rho,W)^{e(\lvert F_1\sqcup F_2\rvert)}.
  \end{align*}
  Hence $\lvert F_1\sqcup F_2\rvert$ is a Sidorenko bigraph by Theorem~\ref{thm:biregularity}.
\end{proofof}

\begin{proofof}{Theorem~\ref{thm:power}}
  The forward direction follows by inductive application of Theorem~\ref{thm:amalgamation}.

  For the reverse direction, let us first prove the case in which $\lvert F\rvert$ is left $d$-regular for
  some $d\in\NN_+$. Given a bigraphon $W\function{\Omega\times\Lambda}{\RR_+}$, we apply Lemma~\ref{lem:main}
  with $\epsilon=1$ to get a bigraphon $W'$ such that $\Delta(F,W')\leq 2\cdot t(\lvert F\rvert,W') = 2\cdot
  t(\lvert F\rvert,W)$. In particular, we have
  \begin{align*}
    t(\lvert F^{\sqcup k}\rvert, W')
    & =
    \int_X t(F,W')(x)^k\ d\mu(x)
    \leq
    2^k\cdot t(\lvert F\rvert, W)^k.
  \end{align*}
  Since $\lvert F^{\sqcup k}\rvert$ is a Sidorenko bigraph, we conclude
  \begin{align*}
    t(\lvert F\rvert, W)
    & \geq
    \frac{1}{2}\cdot t(\lvert F^{\sqcup k}\rvert, W')^{1/k}
    \geq
    \frac{1}{2}\cdot t(\rho, W')^{e(\lvert F^{\sqcup k}\rvert)/k}
    =
    \frac{1}{2}\cdot t(\rho, W')^{e(\lvert F\rvert)},
  \end{align*}
  so $\lvert F\rvert$ is a Sidorenko bigraph by Lemma~\ref{lem:tensor}.

  \medskip

  Let us now show the general case. Let $d\df\Delta_1(\lvert F\rvert)$ and let $\widehat{F}$ be the flag
  obtained from $F$ by adding $d\cdot v_1(\lvert F\rvert) - e(\lvert F\rvert)$ vertices to $V_2(\lvert
  F\rvert)$ and connecting each of these newly added vertices to a single left vertex so that
  $\lvert\widehat{F}\rvert$ is left $d$-regular. By repeated application of Corollary~\ref{cor:deg1}, $\lvert
  F\rvert$ is a Sidorenko bigraph if and only if $\lvert\widehat{F}\rvert$ is so.

  Note now that even though $\lvert\widehat{F}^{\sqcup k}\rvert$ is not left regular, it can also be obtained
  from $F^{\sqcup k}$ by adding vertices to $V_2(F^{\sqcup k})$ and connecting each of them to a single left
  vertex. Again, by Corollary~\ref{cor:deg1}, $\lvert F^{\sqcup k}\rvert$ is a Sidorenko bigraph if and only
  if $\lvert\widehat{F}^{\sqcup k}\rvert$ is so. Since $\lvert\widehat{F}\rvert$ is left $d$-regular, the
  result now follows from the previous case.
\end{proofof}

To prove Theorem~\ref{thm:stars}, we will use steps similar to
Theorem~\ref{thm:biregularity} with the differences that this time we are
only interested in the left regularity, and we focus on $K_{1,d}^L$ rather
than on $e_1$ so that in particular the density of $K_{1,d}$ will be
preserved throughout our transformations. Fortunately, since now we are only
concerned with the left side, the analogue of the crucial
Lemma~\ref{lem:lowerreg} is much easier to prove.

\begin{lemma}\label{lem:lowerK1d}
  If $W\function{\Omega\times\Lambda}{\RR_+}$ is a bigraphon over spaces $\Omega=(X,\mu)$ and
  $\Lambda=(Y,\nu)$ such that $\Delta(K_{1,d}^L,W)\leq 2\cdot t(K_{1,d},W)$, then there exists a bigraphon
  $W'\function{\Omega'\times\Lambda}{\RR_+}$ such that the following hold.
  \begin{enumerate}
  \item We have $t(K_{1,d},W') = t(K_{1,d},W)$.
  \item We have $\delta(K_{1,d}^L,W')\geq t(K_{1,d},W')/6$.
  \item For every bigraph $G$, we have $t(G,W')\leq 3^{v_1(G)}\cdot t(G,W)$.
  \end{enumerate}
\end{lemma}

\begin{proof}
  If $t(K_{1,d},W) = 0$, we can simply take $W'\df W$, so suppose $t(K_{1,d},W) > 0$. Let
  \begin{align*}
    \widehat{X} & \df \{x\in X\mid t(K_{1,d}^L,W)(x) \geq t(K_{1,d},W)/2\}
  \end{align*}
  and note that since $\Delta(K_{1,d}^L,W)\leq 2\cdot t(K_{1,d},W)$, Markov's Inequality implies
  $\mu(\widehat{X})\geq 1/3$.

  Let $\widehat{W}\df W\rest_{\widehat{X}\times Y}$ and note that
  \begin{align*}
    t(G,\widehat{W})
    & \leq
    \frac{1}{\mu(\widehat{X})^{v_1(G)}}\cdot t(G,W)
    \leq
    3^{v_1(G)}\cdot t(G,W);
    \\
    t(K_{1,d},\widehat{W})
    & \geq
    t(K_{1,d},W);
    \\
    \delta(K_{1,d}^L,\widehat{W})
    & \geq
    \frac{1}{2}\cdot t(K_{1,d},W)
    \geq
    \frac{1}{6}\cdot t(K_{1,d},\widehat{W}).
  \end{align*}

  Thus, defining $W'\df (t(K_{1,d},W)/t(K_{1,d},\widehat{W}))^{1/d}\cdot \widehat{W}$ gives the desired result.
\end{proof}

\begin{proofof}{Theorem~\ref{thm:stars}}
  We make the constructions $W\Longrightarrow W_1\Longrightarrow W_2\Longrightarrow W_3$, where the first and
  third arrows are applications of Lemma~\ref{lem:main} both with $F=K_{1,d}^L$ but with $\epsilon=1$ and
  $\epsilon=1/6$, respectively, and the second arrow is an application of Lemma~\ref{lem:lowerK1d}.

  Our constructions ensure that $t(K_{1,d},W_i) = t(K_{1,d},W)$ for every $i\in[3]$. Note also that
  \begin{align*}
    \delta(K_{1,d}^L,W_3)
    & \geq
    \min\left\{t(K_{1,d},W_3),\frac{\delta(K_{1,d}^L,W_2)}{1/6}\right\}
    =
    t(K_{1,d},W_3),
  \end{align*}
  so $W_3$ is $K_{1,d}^L$-regular. Since $t(e_1,W_3)(x) = t(K_{1,d}^L,W_3)(x)^{1/d}$ for every $x$, it follows
  that $W_3$ is left regular, which in particular implies $t(K_{1,d},W_3) = t(\rho,W_3)^d$.

  Then we can deduce the following chain of inequalities.
  \begin{align*}
    t(G,W)
    & \geq
    2^{v_1(G) - e(G)/d}\cdot t(G,W_1)
    \\
    & \geq
    2^{v_1(G) - e(G)/d}\cdot 3^{-v_1(G)}\cdot t(G,W_2)
    \\
    & \geq
    2^{v_1(G) - e(G)/d}\cdot 3^{-v_1(G)}\cdot 7^{v_1(G)-e(G)/d}\cdot t(G,W_3)
    \\
    & \geq
    2^{v_1(G) - e(G)/d}\cdot 3^{-v_1(G)}\cdot 7^{v_1(G)-e(G)/d}\cdot t(\rho,W_3)^{e(G)}
    \\
    & =
    2^{v_1(G) - e(G)/d}\cdot 3^{-v_1(G)}\cdot 7^{v_1(G)-e(G)/d}\cdot t(K_{1,d},W_3)^{e(G)/d}
    \\
    & =
    2^{v_1(G) - e(G)/d}\cdot 3^{-v_1(G)}\cdot 7^{v_1(G)-e(G)/d}\cdot t(K_{1,d},W)^{e(G)/d}.
  \end{align*}
  Since this is true for every bigraphon $W$, by
  Lemma~\ref{lem:tensor} we conclude that $t(G,W)\geq t(K_{1,d},W)^{e(G)/d}$,
  again for every $W$.
\end{proofof}

\section{Reflective tree decompositions}
\label{sec:reftree}

In this section we prove Theorem~\ref{thm:reftree} on reflective tree decompositions.

\begin{proofof}{Theorem~\ref{thm:reftree}}
  Let $\cW$ be the class of biregular bigraphons that are bounded away from zero. We claim that it is
  sufficient to show that
  \begin{align}\label{eq:weakdom}
    \frac{t(G,W)}{t(\rho,W)^{e(G)}} & \geq \frac{t(H,W)}{t(\rho,W)^{e(H)}}
  \end{align}
  only for $W\in\cW$. Indeed, if $W$ is an arbitrary non-zero biregular bigraphon, then it can be approximated
  by $W_\epsilon\df W + \epsilon\in\cW$ ($\epsilon > 0$) and~\eqref{eq:weakdom} for $W$ follows by applying
  the Dominated Convergence Theorem to the same inequality for $W_\epsilon$ as $\epsilon\to 0$.

  Since $\cW$ is closed under tensor powers, by Lemma~\ref{lem:tensor}, it is enough to
  prove~\eqref{eq:weakdom} for every $W\in\cW$ up to a multiplicative constant that does not depend on $W$.

  For convenience of notation, by possibly replacing $W\function{\Omega\times\Lambda}{\RR_+}$ with
  $W'\function{(\Omega\times\Lambda)\times(\Omega\times\Lambda)}{\RR_+}$ given by $W'((x_1,y_1),(x_2,y_2)) =
  W(x_1,y_2)$, we may assume that $\Omega=\Lambda$ (but $W$ is still \emph{not} necessarily symmetric!).

  \medskip

  Given a subtree $T'$ of $T$, let $V_{T'}\df\bigcup_{U\in V(T')} U$, $G_{T'}\df G\rest_{V_{T'}}$ and
  \begin{align*}
    d_{T'} & \df e(G_{T'}) - \sum_{\{U_1,U_2\}\in E(T')} e(\lvert C_2(F_{U_1 U_2}')\rvert),
  \end{align*}
  where $F_{U_1 U_2}'\df (G\rest_{U_1} - E(G\rest_{U_1\cap U_2}), U_1\cap U_2)$ is as in
  Remark~\ref{rmk:reftree}, which guarantees that the summand does not depend on the orientation
  of the edge $\{U_1,U_2\}$.

  Given further a bigraphon $W\function{\Omega\times\Omega}{\RR_+}$ in $\cW$ over a space $\Omega=(X,\mu)$,
  let $f_{T'}\function{\Omega^{V_{T'}}}{\RR_+}$ be given by
  \begin{align*}
    f_{T'}(x)
    & \df
    \frac{t((G_{T'},V_{T'}),W)(x)}{\prod_{\{U_1,U_2\}\in E(T')} t(C_2(F_{U_1 U_2}'),W)(x_{U_1\cap U_2})}.
  \end{align*}
Let us remark that the flag $(G_{T'},V_{T'})$ is trivial (totally labeled) hence the
expression $t((G_{T'},V_{T'}),W)(x)$ is simply equal to $\prod_{(v,w)\in E(G_{T'})}W(x_v,x_w)$,
where we assume that $v\in V_1,\ w\in V_2$ and no integration takes place. For the
sake of uniformity, however, we stick to the former notation.

Note that since $W$ is bounded away from zero, all functions $f_{T'}$ are bounded.

  \begin{claim}\label{clm:intfT}
    For every $U_0\in V(T)$ and every $x\in X^{U_0}$, we have
    \begin{align}\label{eq:intfT}
      \int_{X^{V(G)\setminus U_0}} f_T(x,x')\ d\mu(x')
      & =
      t(\rho,W)^{d_T - e(G\rest_{U_0})}\cdot t((G\rest_{U_0},U_0),W)(x).
    \end{align}
  \end{claim}

  \begin{proof}
    We will show by induction on $v(T) - v(T')$ that if $T'$ is a subtree of $T$ with $U_0\in V(T')$, then
    \begin{align}\label{eq:fT}
      \int_{X^{V(G)\setminus U_0}} f_T(x,x')\ d\mu(x')
      & =
      t(\rho,W)^{d_T -  d_{T'}}\cdot \int_{X^{V_{T'}\setminus U_0}} f_{T'}(x,x')\ d\mu(x').
    \end{align}
    Once this is proved then~\eqref{eq:intfT} follows by taking $T'$ as the subtree of $T$ with
    $V(T')=\{U_0\}$.

    If $T' = T$, then~\eqref{eq:fT} holds trivially. If $T'$ is a proper subtree of $T$ containing $U_0$, then
    let $T''$ be a subtree of $T$ containing $T'$ as a subtree and having exactly one more vertex $U_1$ than
    $T'$, which must necessarily be a leaf of $T''$, so we can let $U_2$ be its unique neighbor in $T'$.
    By inductive
    hypothesis, we have
    \begin{align*}
      \int_{X^{V(G)\setminus U_0}} f_T(x,x')\ d\mu(x')
      & =
      t(\rho,W)^{d_T - d_{T''}}\cdot \int_{X^{V_{T''}\setminus U_0}} f_{T''}(x,x')\ d\mu(x').
    \end{align*}
    But note that in the expression for $f_{T''}(x,x')$, variables indexed by $U_1\setminus V_{T'}$ appear
    only in the numerator, so integrating these in the right-hand side of the above gives
    \begin{multline*}
      \int_{X^{V(G)\setminus U_0}} f_T(x,x')\ d\mu(x')
      \\
      =
      t(\rho,W)^{d_T - d_{T''}}\cdot
      \int_{X^{V_{T'}\setminus U_0}}
      f_{T'}(x,x')
      \cdot
      \frac{t(F_{U_1 U_2}',W)(x_{U_1\cap U_2})}{t(C_2(F_{U_1 U_2}'),W)(x_{U_1\cap U_2})}
      \ d\mu(x').
    \end{multline*}
    Since $W$ is biregular, the fraction under the integral is equal to $t(\rho,W)^{e(\lvert F_{U_1
        U_2}'\rvert) - e(\lvert C_2(F_{U_1 U_2}')\rvert)}$, and~\eqref{eq:fT} follows.
  \end{proof}

Let now
  \begin{align*}
    Z & \df \int_{X^{V(G)}} f_{T}(x)\ d\mu(x).
  \end{align*}
By picking $U_0\in V(T)$ arbitrarily and integrating~\eqref{eq:intfT} over $x$,
  we similarly get
  \begin{equation}\label{eq:z_t}
  Z = t(\rho,W)^{d_T-e(H)}\cdot t(H,W).
  \end{equation}

  We now let $\eta$ be the probability measure such that $d\eta(x) = (f_T(x)/Z)\ d\mu(x)$ and for each
  $\{U_1, U_2\}\in E(T)$, we let
  \begin{align*}
    D_{U_1 U_2}
    & \df
    \begin{multlined}[t]
      \biggl\{x\in X^{U_1\cap U_2} \;\bigg\vert\;
      t(C_2(F_{U_1 U_2}'),W)(x)
      \\
      \leq
      \frac{t(H,W)}{
        2v(T)\cdot t(G\rest_{U_1\cap U_2},W)\cdot
        t(\rho,W)^{e(H) - e(G\rest_{U_1\cap U_2}) - e(\lvert C_2(F_{U_1 U_2}')\rvert)}
      }
      \biggr\};
    \end{multlined}
    \\
    D_{U_1 U_2}' & \df \{x\in X^{V(G)} \mid x_{U_1\cap U_2}\in D_{U_1 U_2}\}.
  \end{align*}

  Then we have
  \begin{align*}
    \eta(D_{U_1 U_2}')
    & =
    \frac{1}{Z}\int_{X^{V(G)}} \One[x_{U_1\cap U_2}\in D_{U_1 U_2}]\cdot f_T(x)\ d\mu(x)
    \\
    & =
    \frac{t(\rho,W)^{e(H) - e(G\rest_{U_1})}}{t(H,W)}
    \int_{X^{U_1}} \One[x_{U_1\cap U_2}\in D_{U_1 U_2}]\cdot t((G\rest_{U_1},U_1),W)(x)\ d\mu(x)
    \\
    & =
    \begin{multlined}[t]
      \frac{t(\rho,W)^{e(H) - e(G\rest_{U_1}) + e(\lvert F_{U_1 U_2}'\rvert) - e(\lvert C_2(F_{U_1 U_2}')\rvert)}}{t(H,W)}
      \\
      \cdot
      \int_{D_{U_1 U_2}}
      t((G\rest_{U_1\cap U_2},U_1\cap U_2),W)(x)\cdot t(C_2(F_{U_1 U_2}'),W)(x)
      \ d\mu(x)
    \end{multlined}
    \\
    & \leq
    \frac{1}{2v(T)},
  \end{align*}
  where the second equality follows from Claim~\ref{clm:intfT} with $U_0=U_1$ and~\eqref{eq:z_t},
  the third equality follows since $W$ is biregular and the inequality follows from the
  definition of $D_{U_1 U_2}$ and the fact that $e(G\rest_{U_1}) - e(\lvert F_{U_1 U_2}'\rvert) =
  e(G\rest_{U_1\cap U_2})$.

  Define then $D\df X^{V(G)}\setminus\bigcup_{\{U_1,U_2\}\in E(T)} D_{U_1 U_2}'$ and note that
  \begin{align*}
    \eta(D) & \geq 1 - \frac{e(T)}{2v(T)} \geq \frac{1}{2}.
  \end{align*}
  We have
  \begin{align*}
    t(G,W)
    & =
    Z \int_{X^{V(G)}} \prod_{\{U_1,U_2\}\in E(T)} t(C_2(F_{U_1 U_2}'),W)(x_{U_1\cap U_2})\ d\eta(x)
    \\
    & \geq
    Z\cdot\eta(D)\cdot
    \prod_{\{U_1,U_2\}\in E(T)}
    \frac{t(H,W)}{
      2v(T)\cdot t(G\rest_{U_1\cap U_2},W)\cdot
      t(\rho,W)^{e(H) - e(G\rest_{U_1\cap U_2}) - e(\lvert C_2(F_{U_1 U_2}')\rvert)}
    }
    \\
    & \geq
    t(\rho,W)^{d_T - e(H)}\cdot t(H,W)\cdot\frac{\eta(D)}{(2v(T))^{e(T)}}
    \prod_{\{U_1,U_2\}\in E(T)}
    t(\rho,W)^{e(\lvert C_2(F_{U_1 U_2}')\rvert)}
    \\
    & \geq
    \frac{1}{2^{e(T)+1}\cdot v(T)^{e(T)}}\cdot
    t(\rho,W)^{e(G) - e(H)}\cdot t(H,W),
  \end{align*}
  where the second inequality follows from~\eqref{eq:z_t} and since $H$ weakly dominates each
  $G\rest_{U_1\cap U_2}$. Therefore~\eqref{eq:weakdom} holds by Lemma~\ref{lem:tensor}, so $G$ weakly
  dominates $H$.

  \medskip

  Finally, if further $H$ is a Sidorenko bigraph, then by Theorem~\ref{thm:biregularity}, $G$ must also be a
  Sidorenko bigraph as it weakly dominates $H$.
\end{proofof}

\section{The symmetric setting}
\label{sec:symmetric}

In this section, we briefly sketch how to adapt the results from Sections~\ref{sec:mainlemma},
\ref{sec:biregularity} and~\ref{sec:reftree} to the symmetric setting. First we note that the tensor power
trick of Lemma~\ref{lem:tensor} still holds in the symmetric setting. For Lemma~\ref{lem:main}, we need to
make some adjustments.

\begin{lemma}[Symmetric version of Lemma~\ref{lem:main}]\label{lem:symmain}
  Let $d\in\NN_+$, let $F=(G,\theta)$ be a left $1$-flag such that $G$ is both left and right $d$-regular and
  let $\epsilon > 0$.

  Then for every graphon $W\function{\Omega\times\Omega}{\RR_+}$ over $\Omega=(X,\mu)$, there exists a graphon
  $W'\function{\Omega'\times\Omega'}{\RR_+}$ such that the following hold.
  \begin{enumerate}
  \item We have $\Delta(F,W')\leq (1+\epsilon)\cdot t(G,W')$.
    \label{lem:symmain:DeltaF}
  \item We have
    \begin{align*}
      \delta(F,W') & \geq \min\left\{t(G,W'), \frac{\delta(F,W)}{\epsilon}\right\}.
    \end{align*}
    \label{lem:symmain:deltaF}
  \item For every bigraph $G'$ with $\max\{\Delta_1(G'),\Delta_2(G')\}\leq d$, we have
    \begin{align*}
      t(G',W') & \geq \left(1 + \frac{1}{\epsilon}\right)^{2e(G')/d - v(G')}\cdot t(G',W).
    \end{align*}
    \label{lem:symmain:tG'lower}
  \item For every bigraph $G'$ with $\min\{\delta_1(G'),\delta_2(G')\}\geq d$, we have
    \begin{align*}
      t(G',W') & \leq \left(1 + \frac{1}{\epsilon}\right)^{2e(G')/d - v(G')}\cdot t(G',W).
    \end{align*}
    \label{lem:symmain:tG'upper}
  \item For every bigraph $G'$ that is both left and right $d$-regular, we have $t(G',W')=t(G',W)$.
    \label{lem:symmain:tG'}
  \end{enumerate}
\end{lemma}

\begin{proof}[Proof (sketch)]
  Analogous to that of Lemma~\ref{lem:main} but using the definition
  \begin{align*}
    W'(x,y) & \df \left(\frac{Z^2}{f(x)f(y)}\right)^{1/d}\cdot W(x,y)
  \end{align*}
  that ensures that $W'$ is symmetric.
\end{proof}

Even though it is possible to adapt the proof of Lemma~\ref{lem:lowerreg} to the symmetric setting, we can
instead simply use the finite version~\cite[Lemma~3.4]{CKLL18a} that inspired it to prove the symmetric version
of the biregularity result, Theorem~\ref{thm:biregularity}.

\begin{theorem}\label{thm:regularity}
  Let $G$ be a bigraph. If there exists $c_G > 0$ such that $t(G,W)\geq c_G\cdot t(\rho,W)^{e(G)}$ for every
  regular graphon $W$, then $G$ is a symmetrically Sidorenko bigraph.
\end{theorem}

\begin{proof}[Proof (sketch)]
  By~\cite[Lemmas~3.3 and~3.4]{CKLL18a}, it is sufficient to show that $t(G,H)\geq c_G'\cdot t(\rho,H)^{e(G)}$
  for some constant $c_G' > 0$ depending only on $G$ and every \emph{graph} $H$ whose degrees are all between
  $\dave(H)/8$ and $2\cdot\dave(H)$, where $\dave(H)$ is the average degree of $H$. By considering the step graphon
  associated with $H$, it follows that it is sufficient to prove that $t(G,W)\geq c_G'\cdot t(\rho,W)^{e(G)}$
  for every graphon $W$ such that
  \begin{align*}
    \frac{t(\rho,W)}{8} & \leq \delta_1(e_1,W) \leq \Delta_1(e_1,W) \leq 2\cdot t(\rho,W).
  \end{align*}

  We then apply Lemma~\ref{lem:symmain} with $\epsilon = 1/8$ and $F=e_1$ to get a graphon $W'$ that satisfies
  \begin{align*}
    \delta_1(e_1,W') & \geq \min\left\{t(\rho,W'), \frac{\delta(e_1,W)}{1/8}\right\} = t(\rho,W),
  \end{align*}
  that is $W'$ is regular. Hence,
  \begin{align*}
    t(G,W)
    & \geq
    9^{2e(G) - v(G)}\cdot t(G,W')
    \geq
    9^{2e(G) - v(G)}\cdot c_G\cdot t(\rho,W')^{e(G)}
    \\
    & \geq
    9^{2e(G) - v(G)}\cdot c_G\cdot t(\rho,W)^{e(G)},
  \end{align*}
  so $G$ is symmetrically Sidorenko.
\end{proof}

Finally, the symmetric analogue of Theorem~\ref{thm:reftree} (i.e., once we also replace weak domination by
its symmetric version) can be shown with the same proof, replacing Theorem~\ref{thm:biregularity} with
Theorem~\ref{thm:regularity} for the final statement on symmetrically Sidorenko bigraphs.

\section{Conclusion and open problems}
\label{sec:conclusion}

In this paper, we have shown how to reduce Sidorenko's Conjecture to biregular
bigraphons (or regular graphons in the symmetric case). We have
also shown that this reduction and the construction of Lemma~\ref{lem:main} can be used to obtain
simple proofs of some properties of the class of Sidorenko bigraphs.

The proofs in Section~\ref{sec:applications} heavily rely on the fact that the two sides of the bigraphs and
bigraphons can be manipulated independently. It is then natural to ask if Theorem~\ref{thm:stars}
holds in the symmetric setting as well (the symmetric analogues of Theorem~\ref{thm:amalgamation} and~\ref{thm:power}
follow from a symmetric analogue of~\cite[Theorem~4]{Sze15b}).

In another direction, Conlon--Kim--Lee--Lee~\cite{CKLL18b} also provided a higher-order version of their
strong tree decompositions, which is reminiscent (but yields a completely different class of symmetrically
Sidorenko bigraphs) of Szegedy's conditionally independent coupling constructions~\cite{Sze15a}. While we
believe that a higher-order version of the reflective tree decompositions result should also hold (more
specifically by using the same definition of higher-order decompositions and simply replacing level $0$ with
reflective tree decompositions), these higher-order decompositions have the restriction that $G\rest_{U_1
  U_2}$ is a forest for each $\{U_1,U_2\}\in E(T)$ and we would like to ask instead if this restriction can be
replaced by some weak domination restriction as in reflective tree decompositions. One stepping stone toward
this goal is the following natural generalization of Theorem~\ref{thm:reftree}.

\begin{conjecture}
  If $T$ is a reflective tree decomposition of a connected non-trivial bigraph $G$ whose core $H$ weakly
  dominates $G\rest_{U_1\cap U_2}$ for every $\{U_1,U_2\}\in E(T)$, then for every non-empty $V\subseteq
  V(T)$, $G$ weakly dominates $G\rest_{\bigcup_{U\in V} U}$.
\end{conjecture}

Theorem~\ref{thm:reftree} is the particular case of the conjecture above when $V$ consists of a single vertex
of $T$.

\section*{Acknowledgment}

We are grateful to Alexander Sidorenko for bringing to our attention the reference~\cite{Sze15b}.

\bibliographystyle{alpha}
\bibliography{refs}

\begin{thebibliography}{CKLL18b}

\bibitem[CFS10]{CFS10}
David Conlon, Jacob Fox, and Benny Sudakov.
\newblock An approximate version of {S}idorenko's conjecture.
\newblock {\em Geom. Funct. Anal.}, 20(6):1354--1366, 2010.

\bibitem[CKLL18a]{CKLL18b}
David Conlon, Jeong~Han Kim, Choongbum Lee, and Joonkyung Lee.
\newblock {Sidorenko's} conjecture for higher tree decompositions.
\newblock Technical Report arXiv:1805.02238 [math.CO], arXiv e-print, 2018.

\bibitem[CKLL18b]{CKLL18a}
David Conlon, Jeong~Han Kim, Choongbum Lee, and Joonkyung Lee.
\newblock Some advances on {S}idorenko's conjecture.
\newblock {\em J. Lond. Math. Soc. (2)}, 98(3):593--608, 2018.

\bibitem[CL17]{CL17}
David Conlon and Joonkyung Lee.
\newblock Finite reflection groups and graph norms.
\newblock {\em Adv. Math.}, 315:130--165, 2017.

\bibitem[ES84]{ES84}
P.~Erd\H{o}s and M.~Simonovits.
\newblock Cube-supersaturated graphs and related problems.
\newblock In {\em Progress in graph theory ({W}aterloo, {O}nt., 1982)}, pages
  203--218. Academic Press, Toronto, ON, 1984.

\bibitem[Hal76]{Hal76}
Rudolf Halin.
\newblock {$S$}-functions for graphs.
\newblock {\em J. Geom.}, 8(1-2):171--186, 1976.

\bibitem[Hat10]{Hat10}
Hamed Hatami.
\newblock Graph norms and {S}idorenko's conjecture.
\newblock {\em Israel J. Math.}, 175:125--150, 2010.

\bibitem[KLL16]{KLL16}
Jeong~Han Kim, Choongbum Lee, and Joonkyung Lee.
\newblock Two approaches to {S}idorenko's conjecture.
\newblock {\em Trans. Amer. Math. Soc.}, 368(7):5057--5074, 2016.

\bibitem[Lov12]{Lov12}
L\'{a}szl\'{o} Lov\'{a}sz.
\newblock {\em Large networks and graph limits}, volume~60 of {\em American
  Mathematical Society Colloquium Publications}.
\newblock American Mathematical Society, Providence, RI, 2012.

\bibitem[LS11]{LS11}
J.~L.~Xiang Li and Bal\'{a}zs Szegedy.
\newblock On the logarithmic calculus and {S}idorenko's conjecture.
\newblock Technical Report arXiv:1107.1153 [math.CO], arXiv e-print, 2011.

\bibitem[Raz07]{Raz07}
Alexander~A. Razborov.
\newblock Flag algebras.
\newblock {\em J. Symbolic Logic}, 72(4):1239--1282, 2007.

\bibitem[RS84]{RS84}
Neil Robertson and P.~D. Seymour.
\newblock Graph minors. {III}. {P}lanar tree-width.
\newblock {\em J. Combin. Theory Ser. B}, 36(1):49--64, 1984.

\bibitem[Sid91]{Sid91}
Alexander Sidorenko.
\newblock Inequalities for functionals generated by bipartite graphs.
\newblock {\em Diskret. Mat.}, 3(3):50--65, 1991.

\bibitem[Sid93]{Sid93}
Alexander Sidorenko.
\newblock A correlation inequality for bipartite graphs.
\newblock {\em Graphs Combin.}, 9(2):201--204, 1993.

\bibitem[Sid21]{Sid21}
Alexander Sidorenko.
\newblock Inequalities for doubly nonnegative functions.
\newblock {\em Electron. J. Combin.}, 28(1):Paper No. 1.32,--16, 2021.

\bibitem[Sim84]{Sim84}
Mikl\'{o}s Simonovits.
\newblock Extremal graph problems, degenerate extremal problems, and
  supersaturated graphs.
\newblock In {\em Progress in graph theory ({W}aterloo, {O}nt., 1982)}, pages
  419--437. Academic Press, Toronto, ON, 1984.

\bibitem[Sze15a]{Sze15a}
Bal\'{a}zs Szegedy.
\newblock An information theoretic approach to {Sidorenko's} conjecture.
\newblock Technical Report arXiv:1406.6738 [math.CO], arXiv e-print, 2015.

\bibitem[Sze15b]{Sze15b}
Bal\'{a}zs Szegedy.
\newblock Sparse graph limits, entropy maximization and transitive graphs.
\newblock Technical Report arXiv:1504.00858 [math.CO], arXiv e-print, 2015.

\end{thebibliography}

\end{document}